\documentclass[a4paper,oneside,11pt]{article}
\usepackage[T1]{fontenc}
\usepackage{amsmath, amsthm,amsfonts,amssymb,amscd,mathrsfs,graphicx,graphics,multicol}
\usepackage[top=2.5cm, bottom=2.5cm, right=2.5cm, left=2.5cm]{geometry}
\usepackage{enumerate}
\usepackage{tikz-cd}
\usepackage{tikz}
\theoremstyle{break}

\newtheorem{definition}[subsection]{Definition}

\newtheorem{remark}[subsection]{Remark}

\newtheorem{claim}[subsection]{Claim}
\newtheorem{theo}[subsection]{Theorem}
\numberwithin{equation}{section}
\newcommand{\PP}{\mathbb{P}}
\newcommand{\F}{\mathcal{F}}

\newcommand{\E}{\mathcal{E}}
\newcommand{\C}{\mathcal{C}}

\DeclareMathOperator{\rk}{rk}

\DeclareMathOperator{\hh}{h}
\DeclareMathOperator{\H^0}{H^0}
\DeclareMathOperator{\HH}{H}
\DeclareMathOperator{\OO}{\mathcal{O}}
\title{An elementary transformation of vector bundles in $\PP^n$}
\author{
Ben Obiero\footnote{The author benefited from a grant from the National Comission of Science, Technology and Innovation of Kenya.},\\
Department of Pure and Applied Mathematics,\\
Technical University of Kenya.\\
}

\date{\today}
\begin{document}
\maketitle
\begin{abstract}
By considering the equivalence between the category of locally free sheaves and the category of algebraic vector bundles, 
we show how elementary transformations of vector bundles can be used to prove a case of the maximal rank hypothesis. We in turn 
show how this can be applied in the study of minimal free resolutions.

\end{abstract}
\section{Introduction}
\par In order to understand the geometry of the ideal $I_S$ of $s$ points in $\PP^n$, one source of information is the 
Hilbert function $h_d(I_S)$, which gives the number of degree $d$ generators of $I_S.$ The relation among these generators 
is captured in the free resolution of $I_S.$ Conjectures have been made as to the prescribed form of the minimal free resolution 
for the ideal of points in $\PP^n.$ One such conjecture is the minimal resolution conjecture by Lorenzini \cite{a} for the 
ideal of points in general position. Hirschowitz and Simpson \cite{h} in their study of minimal free resolutions 
showed that in order to prove that an ideal of points in general position has a minimal free resolution of the expected 
form, it suffices to show that some evaluation map is of maximal rank. 
\newline In this paper, we make use of the equivalence between the category of locally free sheaves and the category of algebraic vector bundles to describe an elementary transformations in $\PP^n$ and explain how this elementary 
transformation can be used to prove a case of maximal rank hypothesis.
\par The paper is organized as follows. In the next section, we give an overview of category theory with a view of 
building notation and developing the language used in the future sections. In section \ref{sheavesandvect} we introduce 
locally free sheaves and relate them to vector bundles. We end the section by defining an elementary transformation 
of algebraic vector bundles. The definitions and properties we come across in these sections follow largely from 
\cite{hart} and \cite{igr}. Our main results are found in section \ref{main}, where we present 
an elementary transformation in $\PP^n,$ and prove that the diagram is indeed a diagram of elementary transformation. 
As a conclusion to this section, we show how our results can be used in studying minimal free resolutions.
\section{Some category theory}
\begin{definition}
A category $\mathcal{C}$ consists of a class of objects, denoted by $Obj(\mathcal{C}),$ together with a 
set of morphisms between any pair $X,Y\in Obj(\mathcal{C})$. The set of morphisms from $X$ to $Y$ is denoted 
by $Hom(X,Y)$ and satisfy the following properties;
\begin{enumerate}[a.]
\item given $f_1\in Hom(X,Y)$ and $f_2\in Hom(Y,Z)$, then the composition $f_2\circ f_1\in Hom(X,Z)$. 
\item for every $X\in Obj(\mathcal{C})$ there exist a morphism $id_X\mapsto Hom(X,X)$ which 
is both right and left identity of the composition.
\item given $f_1\in Hom(X,Y)$, $f_2\in Hom(Y,Z)$ and $f_3\in Hom(Z,W),$ $f_1\circ (f_2 \circ f_3)=(f_1\circ f_2) \circ f_3$ 
For every ordered triple $(X,Y,Z)\in Obj(\mathcal{C})$.
\end{enumerate}
\end{definition}
Given a category $\mathcal{C},$ we can define the \textbf{opposite category} $\mathcal{C}^{op}$ by reversing the sense 
of maps in the category $\mathcal{C}.$ That is, the objects in the category $\mathcal{C}^{op}$ is the same as the objects 
in the category $\mathcal{C}$ and for any pair $X,Y$, $Hom_{\mathcal{C}}(X,Y)=Hom_{\mathcal{C}^{op}}(Y,X)$
\par Given two categories $\mathcal{C}_1$ and $\mathcal{C}_2$ we can define a function 
$F:\mathcal{C}_1\mapsto \mathcal{C}_2$ between the two categories. Such a function is called a \textbf{functor}. 
Functors can be covariant or contravariant.
\newline A \textbf{covariant functor} from a category $\mathcal{C}_1$ to a category $\mathcal{C}_2$  consists of 
a map $F$ from $Obj(\mathcal{C}_1)$ to $Obj(\mathcal{C}_2)$ together with, for every pair $(X,Y)\in Obj(\mathcal{C}_1),$ 
a function $F_{X,Y}:Hom(X,Y)\mapsto Hom(F(X),F(Y)),$ such that $F$ commutes with composition and carries $id_X$ to $id_{F(X)}.$
 \newline A \textbf{contravariant functor} $\mathcal{C}_1$ to a category $\mathcal{C}_2$ is a functor $F$ defined as
$F:\mathcal{C}_1^{op}\mapsto \mathcal{C}_2$. In other words it is a functor that reverses the sense of the morphisms.
\par Given two functors $F_1$ and $F_2$ from $\mathcal{C}_1$ to $\mathcal{C}_2,$ a \textbf{natural transformation} of $F_1$ to $F_2$ consists of, for each $X\in(\mathcal{C}_1),$ a morphism $\phi_X:F_1(X)\mapsto F_2(X)$ such that for every morphism 
$f\in Hom(X,Y),$ the diagram 
\[
\begin{tikzcd}
F_1(X)\arrow{d}{\phi_X}\arrow{r}{F_1(f)}&F_1(Y)\arrow{d}{\phi_Y}\\
F_2(X)\arrow{r}{F_2(f)}&F_2(Y)
\end{tikzcd}
\]
is commutative. A \textbf{natural isomorphism} of functors is a natural transformation for which $\phi_X$ is an 
isomorphism. The data of functors $F:\C_1\mapsto\C_2$ and $F':\C_2\mapsto\C_1$ for which $F\circ F'$ is naturally an 
isomorphism to the identity function on $Id_{\C_2}$ on $\C_2$ and $F'\circ F$ is naturally isomorphic to $Id_{\C_1}$ is 
called an \textbf{equivalence of categories}. We say that two categories are equivalent if there exist an equivalence 
between them.
\par A category $\mathcal{C}$ is said to be \textbf{enriched} over a category $\mathcal{D}$ if for every 
$X,Y\in Obj(\mathcal{C}),\ Hom(X,Y)\in Obj(\mathcal{D})$ and the composition 
$Hom(Y,Z)\times Hom(X,Y)\mapsto Hom(Y,Z)$ is a morphism in $\mathcal{D}.$ 
\newline An abelian category is a category enriched with the category of abelian groups with some extra coditions.
\begin{definition}
An abelian category is a category $\mathcal{C}$, such that: for each $X,Y\in Obj(\mathcal{C}),\ Hom(X,Y)$ 
has a structure of an abelian group, and the composition law is linear; finite direct sums exist; every 
morphism has a kernel and a cokernel; every monomorphism is the kernel of its cokernel, every epimorphism 
is the cokernel of its kernel; and finally, every morphism can be factored into an epimorphism followed by 
a monomorphism.
\end{definition}
The category $\textbf{Ab}$ of abelian groups and the category $\textbf{R-Mod}$ of (left) $R$-modules over a 
ring $R$ are examples of abelian categories.  

Given a category $\mathcal{C},$ one can construct a \textbf{complex}. That is, a sequence
\[
\begin{tikzcd}
\cdots\arrow{r}&X_{i-1}\arrow{r}{\delta_i}&X_i\arrow{r}{\delta_{i+1}}&X_{i+1}\arrow{r} &\cdots \\
\end{tikzcd}
\]
For which $im(\delta_i)\subseteq ker(\delta_{i+1}).$ If $im(\delta_i)=ker(\delta_{i+1})$ for all $i$ then the complex 
is called an \textbf{exact sequence}. The following properties hold for exact sequences in an abelian category.
\begin{theo}[Five lemma]
Consider the commutative diagram below with exact rows.
\[
\begin{tikzcd}
X_0\arrow{d}{f_0}\arrow{r}&X_1\arrow{d}{f_1} \arrow{r}&X_2\arrow{d}{f_2}\arrow{r}&X_3\arrow{d}{f_3}\arrow{r}&X_4\arrow{d}{f_4} \\
Y_0\arrow{r}&Y_1 \arrow{r}&Y_2\arrow{r}&Y_3\arrow{r}&Y_4\\
\end{tikzcd}
\]
\begin{enumerate}[a.]
\item If $f_1$ and $f_3$ are monomorphisms and $f_0$ is an epimorphism, then $f_2$ is a monomorphism.
\item If $f_1$ and $f_3$ are epimorphisms and $f_4$ is a monomorphism, then $f_2$ is an epimorphism.
\end{enumerate}
\end{theo}
\begin{theo}[Snake lemma]
Consider the diagram below whose rows are short exact sequences.
\[
\begin{tikzcd}
0\arrow{r}&X_1\arrow{d}{f_1} \arrow{r}&X_2\arrow{d}{f_2}\arrow{r}&X_3\arrow{d}{f_3}\arrow{r}&0\\
0\arrow{r}&Y_1 \arrow{r}&Y_2\arrow{r}&Y_3\arrow{r}&0\\
\end{tikzcd}
\] 
Then there exists a canonical homomorphism
$\delta:ker(f_3)\mapsto coker(f_1)$ called the connecting homomorphism, such that
\[
\begin{tikzpicture}[baseline= (a).base]
\node[scale=.8] (a) at (0,0){
\begin{tikzcd}
0\arrow{r}&ker(f_1)\arrow{r}&ker(f_2)\arrow{r}&ker(f_3)\arrow{r}{\delta}&coker(f_1)\arrow{r}&coker(f_2)\arrow{r}&coker(f_3)\arrow{r}&0\\
\end{tikzcd}};
\end{tikzpicture}
\]
is exact, where all the maps other than $\delta$ are the obvious ones induced by the diagram.
\end{theo} 
Let $F:\mathcal{C}_1 \longrightarrow \mathcal{C}_2$ be a covariant functor between two abelian categories 
$\mathcal{C}_1$ and $\mathcal{C}_2.$ The functor $F$ is called \textbf{additive} if it commutes with addition of morphisms. 
$Hom$ and the tensor product 
are examples of additive functors.
An additive functor sends complexes to complexes, but does not generally send exact sequences to exact sequences.
\begin{definition}
Let $F$ be a functor;
\begin{enumerate}[a.]
\item $F$ is \textbf{left exact} if for a given exact sequence $0 \longrightarrow X_1\longrightarrow  X_2\longrightarrow X_3$ 
the sequence \\ 
$0 \longrightarrow F(X_1)\longrightarrow  F(X_2)\longrightarrow F(X_3)$ is exact.
\item $F$ is \textbf{right exact} if for a given exact sequence $\longrightarrow X_1\longrightarrow  X_2\longrightarrow X_3 \longrightarrow 0$ the sequence  $ \longrightarrow F(X_1)\longrightarrow  F(X_2)\longrightarrow F(X_3)\longrightarrow 0$ 
is exact.
\item $F$ is \textbf{exact} if it is both left exact and right exact.
\end{enumerate}
\end{definition}  
The $Hom$ functor is left exact while the tensor product functor 
is right exact. 
We now give the definition of cohomological functors. 
\begin{definition}[Cohomological functor]
A cohomological functor (or $\delta$-functor ) between abelian categories $\mathcal{C}_1$ and $\mathcal{C}_2$ is a sequence of functors 
$$T^i : \mathcal{C}_1 \mapsto \mathcal{C}_2\  i= 0,1,\ldots$$
plus for each short exact sequence 
$$0 \longrightarrow X\longrightarrow Y\longrightarrow Z\longrightarrow 0$$
in $\mathcal{C}_1$ a morphism $\delta^i:T^i (Z)\mapsto T^{i+1}(X)$ functorial in the sequence, such that 
the sequence
\[
\begin{tikzpicture}[baseline= (a).base]
\node[scale=.8] (a) at (0,0){
\begin{tikzcd}
0\arrow{r}&T^0(X) \arrow{r}& T^0(Y) \arrow{r}& T^0(Z) \arrow{r}{\delta^0}&T^1(X) \arrow{r}& T^1(Y) \arrow{r}& T^1(Z) \arrow{r}{\delta^1}&T^2(X) \arrow{r}&\cdots\\
\end{tikzcd}};
\end{tikzpicture}
\]
is exact.
\end{definition}
A cohomological functor $T$ is said to be \textbf{universal} if given any other cohomological functor $U$ and a 
natural transformation $f^0:T^0\mapsto U^0,$ there is a unique sequence of natural transformations 
$f^i:T^i\mapsto U^i$ starting with $f^0$ which commute with the $\delta_i$. Given $T^0,$ any two 
extensions of it to a universal cohomological functor are naturally isomorphic.
\par Let $I$ be an object in in an abelian category $\mathcal{C}.$ Then $I$ is \textbf{injective} if the functor\\ 
$Hom(-,I):\mathcal{C}^{op}\mapsto \textbf{Ab}$ is exact. Since the $Hom$ functor is left exact it 
suffices to show that if $0\longrightarrow X \longrightarrow Y$ is a monomorphism, then for any 
morphism $X\mapsto I$ we can find some morphism $Y\longrightarrow I$ so that the diagram below commute.
\[
\begin{tikzcd}
0\arrow{r}&X\arrow{rd}\arrow{r}&Y\arrow[d, dashrightarrow]\\
&&I
\end{tikzcd}
\]
If for every object $X$ in category $\mathcal{C}$ there exists a monomorphism $X\mapsto I,$ where 
$I$ is an injective element, then we say that the category $\mathcal{C}$ has enough injectives. 
For an abelian category with enough injectives, any universal cohomology functor 
can be computed using injective resolutions. Thus it is possible to define the right derived functor 
of $F$ as follows; for any object $X,$ if $I^{\ast}$ is an injective resolution of $X,$ 
set $R^iF(X)=H^i(F(I^{\ast})).$
\section{Sheaves}
\label{sheavesandvect} 
\begin{definition}[Presheaf of abelian groups]
Let $X$ be a topological space.  A presheaf  $\mathcal{F}$ of abelian groups on $X$ is an assignment of 
each open set $U\subset X$ an abelian group $\mathcal{F}(U)$, and to every inclusion $V\subset U$ of open 
subsets of $X$ a morphism of abelian groups $r^U_V: \mathcal{F}(U)\longrightarrow  \mathcal{F}(V)$ called 
the restriction of $U$ to $V$ subject to the following conditions. 
\begin{enumerate}[a.]
\item $\mathcal{F}(U)=0\ \Leftrightarrow \ U=\varnothing,$ the empty set, 
\item  $r^U_U$ is the identity map $\mathcal{F}(U)\longrightarrow  \mathcal{F}(U)$  and 
\item  for any three open sets $U,\ V,$ and $W$ such that $W\subset V\subset U$, $r^U_W =r^V_W\circ r^U_V $
\end{enumerate}
\end{definition}
In other words, a presheaf $\mathcal{F}$ of abelian groups is a contravariant functor 
$\mathcal{F}:\textbf{Top}(X)\mapsto \textbf{Ab},$ where $\textbf{Top}(X)$ is a category whose objects are 
open sets and morphisms inclusion maps in which $Hom(U,V)=\varnothing$ if $V$ is not a subset of $U$ and 
$Hom(U,V)$ has only one element if $V$ is a subset of $U$. The category $\textbf{Ab}$ in the definition of 
presheaves of abelian groups can be replaced by any category $\mathcal{C}$ to obtain a presheaf with values 
in the fixed category $\mathcal{C}.$ We refer to $\mathcal{F}(U)$ as the sections of the presheaf 
$\mathcal{F}$ over the open set $U$. 
\newline If in addition the presheaf in the definition above satisfies the following conditions: 
\begin{enumerate}[i.]
\item  for any open set $U$, if ${V_i}$ is an open cover for $U$ and if $s\in \mathcal{F}(U)$  
is an element such that $r^U_{V_i}(s)=0$ for all $i$, then $s = 0$; 
\item for an open set $U$, if ${V_ i}$ is an open cover for $U$, and if we have elements 
$s_i \in \mathcal{F}(V_i)$ for each $i$, with the property that for each 
$i,j,\ r^{U}_{V_i\cap V_j}(s_i)= r^{U}_{V_i\cap V_j}(s_j)$ then there is an element $s\in \mathcal{F}(U)$  such that $r^{U}_{V_i}(s)=s_i$  for each $i$. 
\end{enumerate}
Then $\mathcal{F}$ is called a \textbf{sheaf}.
 \begin{definition}
Let $\mathcal{F}$ and $\mathcal{G}$  be presheaves on $X.$ A morphism 
$\phi : \mathcal{F}\longrightarrow \mathcal{G}$  consists of a morphism of abelian groups 
$\phi (U):\mathcal{F}(U)\mapsto \mathcal{G}(U)$ for each open set $U$  such that 
whenever $V\subset U$ is an inclusion, the diagram 
\[
\begin{tikzcd}
\mathcal{F}(U)\arrow{r}{\phi(U)}\arrow{d}{r^U_V}&\mathcal{G}(U)\arrow{d}{r'^U_V}\\
\mathcal{F}(V)\arrow{r}{\phi(U)}&\mathcal{G}(V)
\end{tikzcd}
\]
commute. The presheaf kernel of $\phi$, presheaf cokernel of $\phi$, and presheaf image of $\phi$ 
to be the presheaves given by $U\mapsto ker(\phi(U)),\ U\mapsto coker(\phi(U))$ 
and $U\mapsto im(\phi(U))$ respectively. 
\end{definition}
\begin{remark}
If $\phi$ is a morphism of sheaves, then the presheaf kernel of $ \phi$ is a sheaf. However, the 
presheaf cokernel and presheaf image of $\phi$ are not necessarily sheaves.
\end{remark}
Let $X = Spec R$ the prime spectrum of 
$R=k[x_0,\ldots x_n]$ endowed with the zariski topology. In this topology open sets are the distinguished 
open sets $D(f)$ for $f\in R$, where $D(F)$  is the set of all functions in $R$ that do not vanish outside 
the vanishing set of $f$. We can then define a ring $\mathcal{O}_X(U)$ for every open set $U\subset X.$ 
We call $\mathcal{O}_X(U)$ the ring of regular function in the neighbourhood of $U$ and the assignment to 
every open set $U$ the ring $\mathcal{O}_X(U)$ the structure sheaf of $X$ denoted by $\mathcal{O}_X$.
The pair $(X,\mathcal{O}_X)$ is a ringed space. 
Cohomology of sheaves can be defined by taking the derived functors of the global section functor. 
This is possible because  for any ring $R$ every $R$-module is isomorphic to a submodule of some 
injective $R$-module.
\par Let $(X,\mathcal{O}_X)$ be a ringed space and consider the category $\textbf{Mod}(X)$ of 
sheaves of $\mathcal{O}_X $-modules. This category has enough injectives. Consequently then the 
category $\textbf{Ab}(X)$ of sheaves of abelian groups on $X$ has enough injectives.
\begin{definition}
Suppose $X$ is a topological space. Denote by $\Gamma (X,\_)$  the the global section functor 
from $\textbf{Ab}(X)$ to $\textbf{Ab}$. We can then define the cohomology functor 
$\HH^i(X,\_)$ as the right derived functors of $\Gamma (X,\_).$ For any sheaf $\mathcal{F}$, 
the groups $\HH^i(X,\mathcal{F})$ are referred to as the cohomology groups of $\mathcal{F}$.
\end{definition}
For a given short exact sequence
$$0\longrightarrow \F_1\longrightarrow \F_2\longrightarrow \F_3\longrightarrow 0,$$
there is a long exact sequence induced by the cohomology functor given by;
\[
\begin{tikzpicture}[baseline= (a).base]
\node[scale=.7] (a) at (0,0){
\begin{tikzcd}
0\arrow{r}&\HH^0(X,\F_1)\arrow{r}&\HH^0(X,\F_2)\arrow{r}&\HH^0(X,\F_3)\arrow{r}&\HH^1(X,\F_1)\arrow{r}&\HH^1(X,\F_2)\arrow{r}&\HH^1(X,\F_3)\arrow{r}&\cdots
\end{tikzcd}};
\end{tikzpicture}
\]
If $X$ is a noetherian topological space of dimension $n,$ then $\HH^i(X,\mathcal{F})=0$ for all $i >n$ 
and all sheaves of abelian groups $\mathcal{F}$ on $X.$ 
\par The dimension of the cohomology group $\HH^i(X,\mathcal{F})$ is denoted by $\hh^i(X,\mathcal{F}).$
If $X=\PP^n$ and $\mathcal{F}=\Omega_{\PP^n}^{p}(d)$ is the sheaf of $p$-forms of degree $d$, 
then the dimension of the cohomology group $\HH^i(\PP^n,\Omega_{\PP^n}^{p}(d))$ is given by the Botts formula.
\begin{theo}[ Botts formula \cite{cmh}]
\label{bott}
\abovedisplayskip=0pt\relax
\[
\hh^i(\PP^n,\Omega_{\PP^n}^{p}(d)) =
\begin{cases}
\binom{d+n-p}{d}\binom{d-1}{p} &\text{for}\ i=0,\ 0\leq p\leq n,\ d>p\\
1& \text{for}\ d=0, 0\leq p=i\leq n\\
\binom{-d-p}{-d}\binom{-d-1}{n-p}& \text{for}\ i=n\ 0\leq p\leq n,\ d<p-n\\
0&\text{otherwise}
\end{cases}
\]
in particular if $p=0,$ we have
\[
\hh^i(\PP^n,\mathcal{O}_{\PP^n}(d)) =
\begin{cases}
\binom{d+n}{d} &\text{for}\ i=0,\ d\geq 0\\
\binom{-d-1}{-d-1-n}& \text{for}\ d\leq -n-1\\
0&\text{otherwise}
\end{cases}
\]
\end{theo}
Let $X=Spec{R}$ and $\mathcal{O}_X$ be the structure sheaf of $X.$ A sheaf of $\mathcal{O}_X$-modules 
is a sheaf $\mathcal{F}$ on $X$ such that for each open set $U\subset X$, the group $\mathcal{F}(U)$ 
is an $\mathcal{O}_X(U)$-module. An $\mathcal{O}_X$-module which is isomorphic to a direct sum of copies 
of $\mathcal{O}_X$ is called a free $\mathcal{O}_X(U)$-module. If the open sets $U$ such that 
$\mathcal{F}|_U$ is an  $\mathcal{O}_X(U)$-module forms an open cover for the topological space $X$ 
then $\mathcal{F}$ is a locally free sheaf. The rank of $\mathcal{F}$ on is the number of copies of the 
structure sheaf needed whether finite or infinite. For a connected topological space $X$, 
the rank of a locally free sheaf is the same everywhere. A locally free sheaf of rank 1 is also called an 
invertible sheaf.
\begin{definition}
Let $X$ be a variety, a vector bundle over $X$ is a variety $E$ with a map $\pi:E\to X$ such that the following 
conditions hold:
\begin{enumerate}[a.]
\item For each $p\in X,$ we have $\pi^{-1}(p)$ is isomorphic to $\mathbb{A}^n$ for some n.
\item There exists a cover $U_i$ of $X$ such that $\pi^{-1}(U_i)$ is isomorphic to $U_i\times \mathbb{A}^n.$
\end{enumerate}
\end{definition}
Given any locally free sheaves, one can define a vector bundle and conversely. By viewing locally free sheaves as 
vector bundles we can define elementary transformations.
\subsection*{Elementary transformation of Vector bundles \cite{mm}}
Let $\F$ be a vector bundle on $X.$ If we define a surjective map $\psi:\F\longrightarrow \F',$ where 
$\F'$ is a vector bundle on a divisor $X'$ of $X,$ then $\E=ker(\psi)$ is a vector bundle on X. The 
procedure of obtaining $\E$ from $\F$ is called the elementary transformation of $\F$ along $\F'$ and 
is denoted by $\E=elm_{\F'}(\F).$ We call $\E$ the elementary transform of $\F$ along $\F.$ For 
the given $\psi:\F\longrightarrow \F'$, we have the following exact, commutative diagram which is 
called the display of the elementary transformation:
\[
\begin{tikzcd}
&0\arrow{d}&0\arrow{d}&&\\
&\F(-X)\arrow{d} \arrow[equals]{r}&\F(-X)\arrow{d}&&\\
0\arrow{r}
&\E\arrow{r}\arrow{d}{\psi '}&\F\arrow{r}{\psi }\arrow{d}&\F'\arrow{r}\arrow[equals]{d}&0\\
0\arrow{r}
&\F''\arrow{r}\arrow{d}&\F|_{X'}\arrow{r}\arrow{d}&\F'\arrow{r}
&0\\
&0&0&&
\end{tikzcd}
\]
where $\F''$ is the kernel of $\psi|_{X'}$ and $\F(-X) = F\otimes_{\mathcal{O}_x}\mathcal{O}_X(-X)$. The 
leftmost vertical exact sequence gives us the inverse of the given transformation, that is, 
$\F(-X) = elm_{F''}(E).$
\begin{remark}
There is an equivalence of categories between the category of algebraic vector bundles and the category of 
locally free sheaves, given by associating to an algebraic vector bundle $F\mapsto X$, the sum $\F$ of sections 
of $F.$ 
\end{remark}
\section{Main}
\label{main}
In this section, we prove our main result, that is, we use the equivalence in category between the category 
of locally free sheaves and the category of algebraic vector bundles to give an elementary transformation 
in $\PP^n$. Now that $\OO_{X}$-modules form an abelian group, but locally free sheaves along with reasonably natural maps 
between them (those arising as maps of $\OO_X$-modules) do not form an abelian category, we will enlarge 
our notion of nice $\OO_X$-modules to quasi-coherent sheaves. In fact our locally free sheaves have finite rank, 
and so we will talk about coherent sheaves. We will conclude this section by describing how the elementary 
transformation in $\PP^n$ can be used to prove the minimal resolution conjecture, especially when the points 
under consideration are in general position.
\begin{theo}
\label{mainprop}
There exist an elementary transformation of vector bundles on $\PP^n$ comprising of the following exact sequences.
\begin{equation}
\label{etgp}
\begin{tikzcd}
&0\arrow{d}&0\arrow{d}&&\\
&\Omega^{p+1}_{\PP^n}(p+1)\arrow{d} \arrow[equals]{r}
&\Omega^{p+1}_{\PP^n}(p+1)\arrow{d}&&\\
0\arrow{r}
&\mathcal{O}_{\PP^n}^{\oplus \binom{n}{p+1}}\arrow{r}\arrow{d}
&\Omega^{p+1}_{\PP^n}(p+2)\arrow{r}\arrow{d}
&\Omega^{p+1}_{\PP^{n-1}}(p+2)\arrow{r}\arrow[equals]{d}
&0\\
0\arrow{r}
&\Omega^{p}_{\PP^{n-1}}(p+1)\arrow{r}\arrow{d}
&\Omega^{p+1}_{\PP^n|\PP^{n-1}}(p+2)\arrow{r}\arrow{d}
&\Omega^{p+1}_{\PP^{n-1}}(p+2)\arrow{r}
&0\\
&0&0&&
\end{tikzcd}
\end{equation}
\end{theo}
The proof of this theorem above follows from the following set of claims. 
\begin{claim}
\begin{enumerate}[i)]
\item The kernel of the map $\Omega^{p+1}_{\PP^m}\longrightarrow \Omega^{p+1}_{\PP^{m-1}}$ is isomorphic to 
$\mathcal{O}_{\PP^m}(-p-2)^{\oplus \binom{m}{p+1}}$.
\item The sequence 
\begin{tikzpicture}[baseline= (a).base]
\node[scale=.8] (a) at (0,0){
 \begin{tikzcd} 0\arrow{r}
&\Omega^{p}_{\PP^{n-1}}(p)\arrow{r}
&\Omega^{p+1}_{\PP^n|\PP^{n-1}}(p+2)\arrow{r}
&\Omega^{p+1}_{\PP^{n-1}}(p+2)\arrow{r}&0\\
\end{tikzcd}};
\end{tikzpicture}
is exact.
\item The kernel of the map $\mathcal{O}_{\PP^n}^{\oplus \binom{n}{p+1}}\longrightarrow \Omega_{\PP^{n-1}}^{p}(2)$ 
is isomorphic to $\Omega^{p+1}_{\PP^n}(2)$. 
\end{enumerate}
\end{claim}
\begin{proof}
\begin{enumerate}[i)]
\item We prove the first claim by induction on $p.$ For the base case, we set $p=0$ and prove that the kernel 
of the map $\Omega^{p+1}_{\PP^n}\longrightarrow \Omega^{p+1}_{\PP^{n-1}}$ is isomorphic to 
$\mathcal{O}_{\PP^n}(-p-2)^{\oplus \binom{n}{p+1}}$. 
\newline Consider the Euler sequence.
\[
\begin{tikzcd}
0\arrow{r}&\Omega_{\PP^{n}}\arrow{r}&\mathcal{O}_{\PP^n}(-1)^{n+1}\arrow{r}{s_n}&\mathcal{O}_{\PP^n}\arrow{r} &0
\end{tikzcd}
\]
We can construct a commutative diagram of exact sequences on $\PP^m;$
\[
\begin{tikzcd}
0\arrow{r}&\Omega_{\PP^n}\arrow{r}\arrow{d}{e}&\mathcal{O}_{\PP^n}(-1)^{\oplus n+1}\arrow{r}\arrow{d}{b}&\mathcal{O}_{\PP^n}\arrow{r}\arrow{d}{c}&0\\
0\arrow{r}&\Omega_{\PP^{n-1}}\arrow{r}&\mathcal{O}_{\PP^{n-1}}(-1)^{\oplus n}\arrow{r}&\mathcal{O}_{\PP^{n-1}}\arrow{r}&0
\end{tikzcd}
\]
where $\PP^{n-1}$ is identified with the locus of $\PP^n$ where the last coordinate vanish. The rows are 
the Euler sequences, $e$ is the restriction of forms, $c$ is the restriction of functions  and $b$ is the 
restriction on the first $n$ summands of $\mathcal{O}_{\PP^n}(-1)^{\oplus n+1}$ and is $0$ on the last summand.
Considering the kernels of $e,\ b$ and $c$ we get a commutative diagram with exact rows and columns,
\begin{equation}
\label{diag1}
\begin{tikzcd}
&0\arrow{d}&0\arrow{d}&0\arrow{d}&\\
0\arrow{r}&K\arrow{r}\arrow{d}&\mathcal{O}_{\PP^n}(-2)^{\oplus n}\oplus \mathcal{O}_{\PP^n}(-1)\arrow{r}{\alpha}\arrow{d}&\mathcal{O}_{\PP^n}(-1)\arrow{r}\arrow{d}&0\\
0\arrow{r}&\Omega_{\PP^n}\arrow{r}\arrow{d}&\mathcal{O}_{\PP^n}(-1)^{\oplus n+1}\arrow{r}{S_n}\arrow{d}&\mathcal{O}_{\PP^n}\arrow{r}\arrow{d}&0\\
0\arrow{r}&\Omega_{\PP^{n-1}}\arrow{r}\arrow{d}&\mathcal{O}_{\PP^{n-1}}(-1)^{\oplus n}\arrow{r}\arrow{d}&\mathcal{O}_{\PP^{n-1}}\arrow{r}\arrow{d}&0\\
&0&0&0&
\end{tikzcd}
\end{equation}
and we need to prove that $K\cong \mathcal{O}_{\PP^n}(-2)^{\oplus n}.$ It suffices to show that $\alpha$ sends 
the last summand $\mathcal{O}_{\PP^n}(-1)$ of $\mathcal{O}_{\PP^n}(-2)^{\oplus n}\oplus \mathcal{O}_{\PP^n}(-1)$ 
isomorphically to $\mathcal{O}_{\PP^n}(-1)$. As this summand is isomorphic to the codomain, it 
suffices to prove that the restriction of $\alpha$ to the summand is $\mathcal{O}_{\PP^n}(-1)$ is generically 
injective. This is true because the restriction of $S_n$ to the last summand is an isomorphism outside of 
$\PP^{n-1}.$ (homomorphism of line bundles on the complement of $\PP^{n-1}$). As a consequence the first row 
of \ref{diag1} splits and $k\cong \mathcal{O}_{\PP^n}(-2)^{\oplus n}$.
\newline Suppose now that the kernel of the map $\Omega^{p}_{\PP^n}\longrightarrow \Omega^{p}_{\PP^{n-1}}$ induced by 
the restriction of forms is isomorphic to $\mathcal{O}_{\PP^n}(-p-1)^{\oplus \binom{n}{p}}$. Consider the 
exact sequence below obtained from the Euler sequence of $\PP^n$.
\[
\begin{tikzcd}
0\arrow{r}&\Omega_{\PP^{n}}^{p+1}\arrow{r}&\wedge^{p+1}(\mathcal{O}_{\PP^n}(-1)^{n+1}\arrow{r}\arrow[equals]{d}&\mathcal{O}_{\PP^n}^{p}\arrow{r} &0\\
&&\mathcal{O}_{\PP^n}(-p-1)^{\oplus \binom{n+1}{p+1}}&&
\end{tikzcd}
\]
We can then construct the commutative diagram below with exact rows;
\[
\begin{tikzcd}
0\arrow{r}&\Omega_{\PP^n}^{p+1}\arrow{r}\arrow{d}{e}&\wedge^{p+1}(\mathcal{O}_{\PP^n}(-1)^{\oplus n+1})\arrow{r}\arrow{d}{b}&\Omega_{\PP^n}^{p}\arrow{r}\arrow{d}{c}&0\\
0\arrow{r}&\Omega_{\PP^{n-1}}^{p+1}\arrow{r}&\wedge^{p+1}(\mathcal{O}_{\PP^{n-1}}(-1)^{\oplus n})\arrow{r}&\Omega_{\PP^{n-1}}^{p}\arrow{r}&0
\end{tikzcd}
\]
where $e$ and $c$ are restrictions of forms and $b$ is described as follows. Recall that $\PP^{n-1}$ is the 
locus where the last coordinate varnishes and decompose $\mathcal{O}_{\PP^n}(-1)^{\oplus n+1}$ as the direct 
sum of the direct sum of the first $n$ summands and the last summand, that is,\\
$\mathcal{O}_{\PP^n}(-1)^{\oplus n+1}=\mathcal{O}_{\PP^n}(-1)^{\oplus n}\oplus \mathcal{O}_{\PP^n}(-1).$ 
The domain of $b$ is $\wedge^{p+1}\mathcal{O}_{\PP^n}(-1)^{\oplus n+1}=\\
\mathcal{O}_{\PP^n}(-p-1)^{\oplus \binom{n}{p+1}}\oplus \mathcal{O}_{\PP^n}(-p-1)^{\binom{n}{p}}$.
The codomain of $b$ is $\wedge^{p+1}(\mathcal{O}_{\PP^n}(-1)^{\oplus n})=\\
\mathcal{O}_{\PP^n}(-p-1)^{\oplus \binom{n}{p+1}}$. 
Thus $b$ sends $\wedge^p(\mathcal{O}_{\PP^n}(-1)^{\oplus n})\otimes \mathcal{O}_{\PP^n}(-1)$ to zero and its restriction to 
$\wedge^{p+1}(\mathcal{O}_{\PP^n}(-1)^{\oplus n})=\mathcal{O}_{\PP^n}(-p-1)^{\oplus \binom{n}{p+1}}$ is natural.
In particular, the kernel of $b$ is $\wedge^{p+1}(\mathcal{O}_{\PP^n}(-1)^{\oplus n})\otimes \mathcal{O}_{\PP^n}(-1)\oplus \wedge^p(\mathcal{O}_{\PP^n}(-1)^{\oplus n})\otimes \mathcal{O}_{\PP^n}(-1)=\\
\mathcal{O}_{\PP^n}(-p-2)^{\oplus \binom{n}{p+1}}\oplus \mathcal{O}_{\PP^n}(-p-1)^{\oplus \binom{n}{p}}$.
Considering the kernel of $a,\ b$ and $c$ we get a commutative diagram with exact rows and columns;
 \begin{equation}
\label{diag1}
\begin{tikzpicture}[baseline= (a).base]
\node[scale=.7] (a) at (0,0){
\begin{tikzcd}
&0\arrow{d}&0\arrow{d}&0\arrow{d}&\\
0\arrow{r}&K\arrow{r}\arrow{d}&\wedge^{p+1}(\mathcal{O}_{\PP^n}(-1)^{\oplus n})\mathcal{O}_{\PP^n}(-1)\oplus \wedge^p(\mathcal{O}_{\PP^n}(-1)^n)\otimes\mathcal{O}_{\PP^n}(-1)\arrow{r}{\alpha}\arrow{d}&\mathcal{O}_{\PP^n}(-p-1)^{\binom{n}{p}}\arrow{r}\arrow{d}&0\\
0\arrow{r}&\Omega_{\PP^n}^{p+1}\arrow{r}\arrow{d}&\wedge^{p+1}(\mathcal{O}_{\PP^n}(-1)^{\oplus n})\oplus\wedge^p(\mathcal{O}_{\PP^n}(-1)^{\oplus n})\otimes\mathcal{O}_{\PP^n}(-1)\arrow{r}{\gamma}\arrow{d}&\Omega^p_{\PP^n}\arrow{r}\arrow{d}&0\\
0\arrow{r}&\Omega_{\PP^{n-1}}^{p+1}\arrow{r}\arrow{d}&\wedge^{p+1}(\mathcal{O}_{\PP^{n-1}}(-1)^{\oplus n})\arrow{r}\arrow{d}&\Omega^p_{\PP^{n-1}}\arrow{r}\arrow{d}&0\\
&0&0&0&
\end{tikzcd}};
\end{tikzpicture}
\end{equation}
It suffices to prove that the first row splits. To do this, we prove that the restriction of 
$\alpha$ to $\wedge^p(\mathcal{O}_{\PP^{n-1}}(-1)^{\oplus n})\otimes\mathcal{O}_{\PP^{n}}(-1)$ is an isomorphism 
on $\mathcal{O}_{\PP^{n}}(-p-1)^{\oplus \binom{n}{p}}.$ Since we know that 
$\wedge^p(\mathcal{O}_{\PP^n}(-1)^n)\otimes\mathcal{O}_{\PP^n}(-1)$ is isomorphic to 
$\mathcal{O}_{\PP^n}(-p-1)^{\oplus \binom{n}{p}}$, it suffices to show that $\alpha$ is of maximal rank on 
all open subsets. This is because on the complement of $\PP^{n-1}$ the morphism $\gamma$ injects 
$\wedge^p(\mathcal{O}_{\PP^n}(-1)^n)\otimes\mathcal{O}_{\PP^n}(-1)$ in $\Omega^p_{\PP^n}.$
\item To prove the exactness of the bottom row,  
consider the exact sequence
\[
\begin{tikzcd}
0\arrow{r}&\mathcal{O}_{\PP^n}(-1)\arrow{r}&\Omega^{1}_{{\mathbb P}^{n}}|_{{\mathbb P}^{n-1}}\arrow{r}&\Omega^{1}_{{\mathbb P}^{n-1}}\arrow{r}&0.\\
\end{tikzcd}
\]
Taking the $(p+1)$-th exterior product and tensoring by $\mathcal{O}(p+2)$ we get the desired exact sequence.
\item Finally, to show that the kernel of the map $\mathcal{O}_{\PP^n}^{\oplus \binom{n}{p+1}}\longrightarrow \Omega_{\PP^{n-1}}^{p}(2)$ 
is isomorphic to $\Omega^{p+1}_{\PP^n}(2)$, we first recall that given a commutative diagram whose lines are two exact sequences.
\[
\begin{tikzcd}
0\arrow{r}&A\arrow{r}\arrow{d}&B\arrow{r}\arrow{d}&C\arrow{r}\arrow{d}&0\\
0\arrow{r}&D\arrow{r}&E\arrow{r}&F\arrow{r}&0\\
\end{tikzcd}
\]
there always exists a map $A\longrightarrow D$ still making the extended diagram commutative. Let $a,\ b$  and $c$ be the 3 
vertical maps appearing in this extended diagram. By the Snake lemma, we have an exact sequence
 \[
\begin{tikzpicture}[baseline= (a).base]
\node[scale=.6] (a) at (0,0){
\begin{tikzcd}
0\arrow{r}&ker(a)\arrow{r}&ker(b)\arrow{r}&ker(c)\arrow{r}&coker(a)\arrow{r}&coker(b)\arrow{r}&coker(c)\arrow{r}&0.\\
\end{tikzcd}};
\end{tikzpicture}
\]
Since $c$ is an isomorphism, $ker(a)=ker(b)$. As the second column of the diagram is obtained by tensoring
$\Omega^{p+1}_{\mathbb{P}^n}(p+2)$ with the exact sequence
\[
\begin{tikzcd}
0\arrow{r}&I\arrow{r}&\mathcal{O}_{{\mathbb P}^{n}}\arrow{r}&\mathcal{O}_{{\mathbb P}^{n-1}}\arrow{r}&0\\
\end{tikzcd}
\]
defining the ideal of $\mathbb{P}^{n-1}$ in $\mathbb{P}^{n},$ we have $ker(a)=ker(b)=\Omega^{p+1}_{{\mathbb P}^{n}}(p+1).$
\end{enumerate}
\end{proof}
\subsubsection*{conclusion}
In conclusion, lets see how the results above can be applied in the study of minimal free resolution.
\par Suppose $X$ is a smooth projective variety and $X'$ is a non-singular divisor of $X.$ Let $\F$ be a locally free sheaf 
on $X$ and
$$0\longrightarrow \F''\longrightarrow \F_{| X'}\longrightarrow \F'\longrightarrow 0$$
be an exact sequence of locally free sheaves on $X'.$ The kernel $\E$ of $\F\longrightarrow \F'$ is a locally free sheaf
on $X$ and we have another exact sequence of locally free sheaves on $X'$
$$0\longrightarrow \F' (-X')\longrightarrow \E_{| X'}\longrightarrow \F''\longrightarrow 0$$
and as well exact sequences of coherent sheaves on X
\begin{equation*}
\label{1}
0\longrightarrow \E \longrightarrow \F \longrightarrow \F'\longrightarrow 0
\end{equation*}
and
\begin{equation*}
\label{2}
0\longrightarrow \F(-X) \longrightarrow \E \longrightarrow \F'' \longrightarrow 0.
\end{equation*}
\begin{theo} [Differential method of Horace]
\label{dmh}
Suppose we are given a surjective morphism of vector spaces,
$\lambda :\H^0\left(X',\F'\right)\longrightarrow L$
and suppose that there exist a point $Z'\in X'$ such that
$\H^0\left(X',\F'\right)\hookrightarrow L\oplus \F'|_{Z'}$
and suppose that $\HH^1\left(X,\E\right)=0.$ Then there exist a quotient 
$\E(Z')\longrightarrow D(\lambda)$ with a kernel contained in $\F'(Z')$ 
of dimension $dim(D(\lambda))=\rk (\F)-dim(ker\lambda)$ having the following property.
Let $\mu :\H^0\left(X,\F\right)\longrightarrow M$ be a morphism of vector spaces, then 
there exist $Z\in X'$ such that if 
$\H^0\left(X,\E\right)\rightarrow M\oplus D(\lambda)$ is of maximal rank then 
$\H^0\left(X,\F\right)\longrightarrow M\oplus L\oplus \F(Z)$ 
is also of maximal rank.
\end{theo}
\begin{remark}
\label{remdmh}
The idea of the theorem is illustrated in the diagram below.
\[
\begin{CD}
0@>>>\H^0(X,\E)@>>>\H^0(X,\F)@>>>\H^0 (X',\F')@>>>0\\
@.@V\alpha_1VV@V\alpha_2VV@V\alpha_3VV@.\\
0@>>>M\oplus D(\lambda)@>>>M\oplus L \oplus \F(Z)@>>>L\oplus D'(\lambda)|_{Z}@>>>0
\end{CD}
\]
The key point is that if the map $\alpha_3$ is bijective, then $\alpha_2$ will be bijective provided 
that $\alpha_1$ is bijective.
\end{remark}
The elementary transformation in theorem \ref{mainprop} above can be used in proving inductively that the map 
\begin{equation}
\label{maxranpn}
\H^0 \left(\PP^n,\Omega^{p+1}_{\PP^n}(d+p+1)\right)\longrightarrow \bigoplus _{i=1}^{s}\Omega^{p+1}_{\PP^n}(d+p+1)|_{P_i}
\end{equation} 
is of maximal rank for a fixed $p$, for all non negative integers $d\geq m.$ To see this, 
set $X=\PP^n,\ X'=\PP^{n-1},\ \F =\Omega^{p+1}_{\PP^n} \F'=\Omega^{p}_{\PP^{n-1}}$ and 
$\E =\mathcal{O}_{\PP^n}(2)^{\oplus \binom{n}{p+1}}$ in the diagram under remark \ref{remdmh}. We can also construct a 
diagram similar to this using the sequence 
\begin{equation*}
\label{2}
0\longrightarrow \H^0(\PP^n,\Omega^{p+1}_{\PP^{n-1}}(d)) \longrightarrow \H^0(\PP^n,\mathcal{O}_{\PP^n}(2)^{\oplus \binom{n}{p+1}}(d-1)) \longrightarrow \H^0(\PP^{n-1},\Omega^{p}_{\PP^{n-1}}) \longrightarrow 0.
\end{equation*}
Call the three vertical maps in this second diagram $\alpha_1',\ \alpha_2'$ and $\alpha_3'.$ By construction of the 
map $\alpha_1$ coincides with the map $\alpha_2'.$ In reference to remark \ref{remdmh}, we have that $\alpha_2$ is 
bijective provided $\alpha_1$ is, and $\alpha_1$ is bijective provided that $\alpha_1'$ is. Bijectivity of $\alpha_2$ 
is a statement on forms of degree $d+1$ and bijectivity of $\alpha_1'$ is a statement of forms of degree $d$. It 
therefore follows that if we can prove such bijectivity for $d=m,$ and also prove that the bijectivity of $\alpha_1$ 
implies bijectivity of $\alpha_2$ and bijectivity of $\alpha_1'$ implies bijectivity of $\alpha_1,$ then it will 
follow by induction that the map \ref{maxranpn} is of maximal rank for all $d\geq m.$ 


\begin{thebibliography}{99}
\bibitem {hart} Hartshorne R (1977) Algebraic Geometry GTM 52, Springer verlag.
\bibitem {h} Hirschowitz, A. and Simpson, C. (1996); \textit{La r\'esolution minimale de l'id\'eal d'un arrangement g\'en\'eral d'un grand nombre de points dans $\PP^n$}.
\bibitem {a} Lorenzini, A. (1993) \textit{The Minimal Resolution Conjecture,} J. Algebra \textbf{156}, 5-35.
\bibitem {mm} Maruyama, M. $(1982);$ \textit{On a family of algebraic vector bundles}, http://dx.doi.org/10.14989/doctor.r2072 
\bibitem {cmh} Okonek C.,Schneider M. and Spindler H. (1980);Vector Bundles on Complex Projective Spaces
\bibitem {igr} Shafarevich I. (1986); Algebraic Geometry II:Cohomology of Algebraic Varieties. Algebraic Surfaces, Encyclopaedia of Mathematical Sciences Volume 35.
\end{thebibliography}
\end{document}